\documentclass[12pt, twoside]{article}
\usepackage{amsmath,amsthm,amssymb}
\usepackage{times}
\usepackage{enumerate}

\def\titlerunning#1{\gdef\titrun{#1}}
\makeatletter
\def\author#1{\gdef\autrun{\def\and{\unskip, }#1}\gdef\@author{#1}}
\def\address#1{{\def\and{\\\hspace*{18pt}}\renewcommand{\thefootnote}{}%
\footnote {#1}}%
\markboth{\autrun}{\titrun}}
\makeatother
\def\email#1{e-mail: #1}
\def\subjclass#1{{\renewcommand{\thefootnote}{}%
\footnote{\emph{Mathematics Subject Classification (2010):} #1}}}
\def\keywords#1{\par\medskip
\noindent\textbf{Keywords.} #1}

\newtheorem{theorem}{Theorem}[section]
\newtheorem{definition}[theorem]{Definition}
\newtheorem{remark}[theorem]{Remark}

\newtheorem{proposition}[theorem]{Proposition}
\newtheorem{lemma}[theorem]{Lemma}
\newtheorem{corollary}[theorem]{Corollary}

\theoremstyle{definition}

\numberwithin{equation}{section}

\begin{document}


\baselineskip=17pt


\titlerunning{Inverses of generators}

\title{Inverses of generators of integrated fractional resolvent operator functions}

\author{Miao Li
\and
Javier Pastor
\and
Sergey Piskarev}

\date{\today}

\maketitle

\address{M. Li: Department of Mathematics, Sichuan University, Chengdu 610064, China; \email{mli@scu.edu.cn}
\and
J. Pastor: Department of Mathematics, Universitat de Val\`{e}ncia, Dr. Moliner 50, 46100 Burjassot, Valencia, Spain; \email{e-mail: pastorv@uv.es }
\and
S. Piskarev: Scientific Research Computer Center, Lomonosov Moscow State University, Vorobjevy Gory, Moscow 119899, Russia; \email: {piskarev@gmail.com}}

\subjclass{Primary 34A08, 47D06, 47D09, 47D62, 47D99; Secondary 26A33, 45N05}


\begin{abstract}
This paper is devoted to the inverse generator problem in the setting of generators of integrated resolvent operator functions. It is shown that if the operator $A$ is the generator of a tempered $\beta$-times integrated $\alpha$-resolvent operator function ($(\alpha,\beta)$-ROF) and it is injective, then the inverse operator $A^{-1}$ is the generator of a tempered $(\alpha,\gamma)$-ROF for all $\gamma > \beta+1/2,$ by means of an explicit representation of the integrated resolvent operator function based in Bessel functions of first kind.  Analytic resolvent operator functions are also considered, showing that $A^{-1}$ is in addition the generator of a tempered $(\delta,0)$-ROF for all $\delta<\alpha.$ Moreover, the optimal decay rate of $(\alpha,\beta)$-ROFs  as $t\to \infty$ is given. These result are applied to fractional Cauchy problem unsolved in the fractional derivative.

\keywords{Inverse operator, $\alpha$-times resolvent families, integrated resolvent operator function, integrated semigroups, integrated cosine operator functions, generators, abstract fractional Cauchy problem, well-posedness}
\end{abstract}

\section{Introduction}

\setcounter{section}{1}
\setcounter{equation}{0}
\setcounter{theorem}{0}

R. deLaubenfels posed the following question (adapted version) at the end of the paper \cite{dL1}:\\

\textsl{Suppose the operator $A$ on a Banach space $X$ generates a bounded $C_0$-semigroup and has dense range. Does $A^{-1}$ generate a bounded $C_0$-semigroup or at least generate a $C_0$-semigroup?} \\

It is shown in \cite{dL1} that the answer is positive for generators of bounded analytic $C_0$-semigroup by means of an explicit representation of the semigroup generated by $A^{-1}.$ In fact this follows from the nice property that the class of sectorial operators is stable under inversion. For other classes of semigroups, as multiplication semigroups or contraction semigroups in a Hilbert space, the answer is positive again since the condition which characterize to be the generator of that sort of semigroup is easily shown to be stable under inversion.
When the semigroup generated by $A$ is exponentially stable, then $A^{-1}$ is a bounded operator and accordingly it generates a  $C_0$-semigroup, but in general it is not uniformly bounded. This situation is analysed by an explicit representation of the semigroup generated by $A^{-1}$ in \cite{Z1,Z2}, including growth estimates. In \cite{G} a sufficient condition on the resolvent map of $A$ under which $A^{-1}$ is the generator of a bounded  $C_0$-semigroup is provided. Several equivalences for $A^{-1}$ generating a $C_0$-semigroup are given in \cite{dL3}.

In general the answer to the question above is negative. H. Komatsu gives in \cite{K} an example of a infinitesimal generator of a contraction semigroup on $c_0$ with dense range, for which the inverse operator is not an infinitesimal generator (see also \cite[\S1.3.6]{MS}). In \cite{GZT} it is shown that the answer is negative in general for $ X=\ell^p$, $ 1<p<\infty,$ $p\ne 2$. Recently a negative answer to the inverse generator problem in $ X=L^p(\mathbb{R})$, $ 1<p<\infty,$ $p\ne 2$, by means of the generator of the shift group, is given in \cite{F}. For other counterexamples see \cite{dL1,dL3,Z1}.

When $ X$ is a Hilbert space the question is still open, although there is a strong negative step in this direction. It is proved in \cite{GZT} that there exist $C_0$-semigroups $ \{e^{tA}\}_{t\ge0}$, of arbitrarily slow growth at infinity such that the densely defined operator $ A^{-1}$ is not the generator of a $C_0$-semigroup.

In the general case, it has been shown in \cite{PZ} that $A^{-1}$ is the generator of a once-integrated semigroup whenever $A$ is injective and generates a bounded $C_0$-semigroup ($A^{-1}$ is the generator of a regularized semigroup, see \cite{dL3}).

As a further motivation,  we mention that there is an interesting relationship between both $A$ and $A^{-1}$ generating bounded $C_0$-semigroups and the power boundedness of the Cayley transform of $A$ (\cite{PZ}). In addition, in \cite{Z1} we can find several applications of the posed problem to systems theory and numerical analysis.

The so called fractional calculus provides an excellent tool for the description of memory and hereditary properties of some materials and physical processes. During the last decades notable contributions have been made to the development of the theory of fractional calculus (\cite{KST, MSK}). The use of fractional order derivatives instead of integer order ones in different fields as viscoelasticity, heat conduction
in materials with memory or in electrodynamics with memory, has developed the interest of researchers in the study of fractional differential equations.

E. Bajlekova has studied systematically fractional evolution equations with Caputo fractional derivatives in the fundamental work \cite{B}, by considering them as equivalent abstract Volterra equations. In a natural way the notion of resolvent family, or operator solution, and its generator arises in the mentioned study (see also \cite{P}). The more general framework of integrated fractional resolvent operator function, which includes integrated semigroups and integrated cosine operator functions, is an appropriate tool to deal with fractional Cauchy problems (\cite{AM, CL,KLW, LS}).

Our aim in this paper is to study the inverse generator problem for the class of generators of integrated fractional resolvent families. The main result of this paper is obtained by Laplace transform techniques and can be summarize in the following terms: if $A$ is the generator of a tempered $\beta$-times integrated $\alpha$-resolvent operator function and is injective, then the inverse operator $A^{-1}$ is the generator of a tempered $\gamma$-times integrated $\alpha$-resolvent operator function for $\gamma > \beta+1/2,$ and it is also the generator of a tempered $\delta$-times resolvent operator function for $\delta<\alpha.$ Explicit representation of the fractional resolvent operator functions generated by $A^{-1},$ in terms of the fractional resolvent operator functions generated by $A$ and Wright functions, are provided. On the other hand, it was shown in \cite{LZ} that an $(\alpha,0)$-ROF can not decay exponentially for $\alpha \in (0,1)$. We will show that none of $(\alpha,\beta)$-ROF is exponentially stable except $\alpha=1$ and $\beta=0$. And the optimal decay rate of $(\alpha,\beta)$-ROF  as $t\to \infty$ is given.

The plan of the paper is as follows. We start by recalling some definitions from fractional calculus, and collecting some of the main properties of certain special functions and of sectorial operators (Section 2). In Section 3 we remind the basic concepts regarding resolvent families. Some results of possible independent interest regarding integrated resolvent families are obtained. We pay special attention to the tempered integrated resolvent families. The main results of this paper are presented in Section 4. Finally, we apply our results to certain fractional Cauchy problems unsolved in the fractional derivative on the spaces $L^p(\mathbb{R}^n)$ and $C_0(\mathbb{R}^n).$

Throughout this paper, $(X, \|\cdot\|)$ denotes a complex Banach space and $A:D(A)\subseteq X \rightarrow X$ stands for a closed linear operator on $X$. As usual, $D(A), $ $R(A),$ $\rho (A)$ and $\sigma(A)$ denote the domain, the range,  the resolvent set  and the spectrum set of $A,$ respectively. Moreover, we denote by $B(X)$ the space of all bounded linear operators on $X.$ We will denote open sectors by $\Sigma _{\theta }:=\{z\in \mathbb{C}:z\neq 0 \text{ and} \left\vert \arg (z)\right\vert <\theta \}$ for $0<\theta \leq \pi ,$ and the line by $\Sigma _{0}:=(0,\infty)$.

\setcounter{equation}{0}

\section{Preliminaries}
Let us recall the basic definitions of fractional calculus (see \cite{KST, MSK}). The fractional integral of order $\alpha>0$ is defined by
$$
\left( I_{0+}^\alpha f\right) (t):=(g_{\alpha}\ast f)(t)= \int_0^t g_{\alpha}(t-s)f(s)\,ds \quad (t>0),
$$
where
\begin{equation*}
g_{\alpha}(t):=
  \left\{
\begin{array}{ll}
\frac{t^{\alpha-1}}{\Gamma(\alpha)},&  t>0, \\
0,          &              t\leq0,
\end{array}\right.
\end{equation*}
and $\Gamma(\alpha)$ is the Gamma function. Set moreover $g_0(t):=\delta(t)$, the Dirac delta-function.
The Riemann-Liouville derivative of order $\alpha>0$ is
$$
\left( D_{0+}^\alpha f\right) (t):=\left( \frac{d}{dt}\right)^m (I_{0+}^{m-\alpha }f)(t),
$$
where $m=\lceil\alpha\rceil$ is the smallest integer greater than or equal to $\alpha,$ and the Caputo fractional derivative of order $\alpha>0$ is defined by
$$
(\mathbf{D}_{t}^{\alpha} f)(t):=\left( D_{0+}^\alpha f\right) (t)-
\sum\limits_{k=0}^{m-1}\frac{f^{(k)}(0)}{\Gamma (k-\alpha +1)}t^{k-\alpha}.
$$
On smooth functions $f(\cdot)$ one has $(\mathbf{D}_{t}^{\alpha}f)(t)= I_{0+}^{m-\alpha}f^{(m)}(t)$.

 In Section \ref{MS} we shall use some special functions. The Wright function $\phi(\rho,\mu;z)$ is defined by the series
   \begin{equation*}
\phi(\rho,\mu;z):=\sum_{k=0}^{\infty}\frac{z^k}{k!\,\Gamma(\rho k+\mu)} \quad (\rho >-1,\; \mu,z\in\mathbb{C} ),
  \end{equation*}
which is convergent in the whole $z$-complex plane. We will be interested in the case $0<\rho \leq 1.$ The following Laplace transform
     \begin{equation} \label{LTW}
\int_0^\infty e^{-\lambda s} s^{\nu\rho} \phi(\rho,1+\nu\rho,-ts^{\rho})\;ds = \lambda^{-1-\nu\rho} \text{exp}(-t\lambda^{-\rho}),
     \end{equation}
for $1+\nu\rho\geq 0,\text{Re}\lambda>0,\, t>0$,  can be obtained from the power series representation of $\phi$ and changing the order of summation and integration.

In \cite{L} is obtained the indicator function of the Wright function for $\rho> 0,$
\begin{equation*}
h(\theta):= \limsup_{r\rightarrow +\infty}\frac{\log|\phi(\rho,\mu;r\text{e}^{i\theta})|}{r^q}=\sigma \cos(q\theta)\quad (|\theta|\leq\pi),
\end{equation*}
where $q:=1/(1+\rho)$ and $\sigma:=(1+\rho)\rho^{-\rho/(1+\rho)}$ are the order and the type of the entire function $\phi(\rho,\mu;z)$ as a function of $z,$ respectively. Therefore, for $0<\rho <1$, as $q >1/2,$ we obtain that there exist positive constants $a,L$ such that
\begin{equation}\label{b_W}
  |\phi(\rho,\mu;-r)|\leq L \text{e}^{-ar^q} \quad (r\geq0).
\end{equation}

Next we recall that the Bessel function of first kind of order a nonnegative real number $\nu$, which we shall denote by $J_\nu,$ are defined by means of
  $$J_\nu(z):=\left(\frac{z}{2}\right)^\nu \sum_{k=0}^{\infty}\frac{(-1)^k (z/2)^{2k}}{k! \, \Gamma(\nu+k+1) }=\left(\frac{z}{2}\right)^\nu \phi(1,1+\nu,-z^{2}/4).$$

  From the asymptotic expansion of $J_\nu$ we know that $J_\nu(r) = O(r^{-1/2})$ as $r \rightarrow +\infty$ (\cite [\S 7.21]{W}). Therefore, $J_{1+\nu}(2\sqrt{r}) r^{-\delta}\in L^1(0,\infty),$ for all $3/4<\delta\leq (3+\nu)/2.$

Let $\beta\geq 0.$ Taking $\rho=1$ and $\nu=1+\beta$ in (\ref{LTW}) we obtain the following Laplace transform ( see also \cite[\S13.3 (4)]{W})

\begin{equation}\label{LT}
\int_0^\infty \text{e}^{-\lambda s} J_{1+\beta}(2\sqrt{st})\, s^{\frac{1+\beta}{2}}\, ds = t^{\frac{1+\beta}{2}}\lambda^{-2-\beta}\text{e}^{-t/\lambda},
\end{equation}
for all $\text{Re}\lambda>0,\,t>0.$

We now turn to consider sectorial operators. For additional information about sectorial operators see \cite{Ha,MS}.
\begin{definition}
We say that a linear operator $A$ is sectorial of angle $\omega$ with $\omega\in [0,\pi
)$ if $\sigma (A)\ \subseteq\overline{\Sigma _{\omega}}$, and for every $\omega ^{\prime }\in (\omega,\pi )$%
\begin{equation*}
\sup \{\| z(z-A)^{-1}\| :z\in \mathbb{C}\backslash \overline{%
\Sigma _{\omega^{\prime }}}\}<\infty .
\end{equation*}%
In short, $A\in \text{Sect}(\omega )$.
\end{definition}

It is worthwhile to highlight that the class of sectorial operators is stable under inversion. Let $A\in \text{Sect}(\omega )$ and assume that $A$ is injective. Then for all $\lambda \in \rho (A)$, $%
\lambda \neq 0$, we have $\lambda ^{-1}\in \rho (A^{-1})$ with
\begin{equation*}
(\lambda ^{-1}-A^{-1})^{-1}=\lambda (1-\lambda (\lambda -A)^{-1}).
\end{equation*}
Therefore, $A^{-1}\in \text{Sect}(\omega )$.

For sectorial operators we have $\ker (A)\cap \overline{R(A)}=\{0\}$. Hence, if $%
\overline{R(A)}=X$, then $A$ is injective. The converse is also true in reflexive Banach spaces.

\section{Integrated fractional resolvent operator functions}\label{ifrof}

Let us recall the definition of the $\alpha$-resolvent operator function for the abstract Cauchy problem
\begin{equation}\label{FCP}
(\mathbf{D}_{t}^{\alpha} u)(t)=Au(t), \quad t>0;  \quad u(0)=x,
\end{equation}
where $\mathbf{D}_{t}^{\alpha}$ is the Caputo derivative of order $\alpha>0$ (\cite{B}).

The Cauchy problem (\ref{FCP}) is well-posed if and only if the following
Volterra integral equation $$\label{1.3}
u(t)=x+\int_{0}^{t}g_{\alpha}(t-s)Au(s){d}s\quad (t\geq 0),$$
is well-posed in the sense of \cite[Definition 1.2]{P}.

\begin{definition}
Let $\alpha>0.$ A function  $R_{\alpha}:[0,+\infty)\rightarrow B(X)$ is called an
$\alpha$-resolvent operator function ($\alpha$-ROF for short) generated by $A,$ if the following
conditions are satisfied:

    \begin{enumerate}[(i)]
     \item $R_{\alpha}(t)$ is strongly continuous for $t\geq0$ and $R_{\alpha}(0)=I,$ where $I$ is the identity on $X$;
    \item $R_{\alpha}(t)D(A)\subset D(A)$ and $AR_{\alpha}(t)x=R_{\alpha}(t)Ax$ for all $x\in D(A)$, $t\geq0$;
    \item for $x \in D(A)$, the resolvent equation
    $$ R_\alpha(t)x = x + \int_0^t g_\alpha(t-s) R_\alpha(s)Ax ds,$$
    holds for all $t\geq 0.$
    \end{enumerate}
    The $\alpha$-ROF $R_{\alpha }$ is called analytic if it admits an analytic extension to a sector $\Sigma _{\theta_{0}}$ for some $\theta _{0}\in (0,\pi /2]$. The analytic $\alpha$-ROF $R_{\alpha }$ is said to be bounded if for each $\theta \in (0,\theta _{0})$ there exists a constant $M_\theta\geq 1$ such that
    \begin{equation*}
    \left\Vert R_{\alpha }(z)\right\Vert \leq M_\theta \quad (z\in\Sigma_{\theta }),
    \end{equation*}
    and we will write in short $A\in\mathcal{A}_{\alpha }(\theta _{0})$.
\end{definition}

For a densely defined operator $A,$ it is shown in \cite{B} that the  Cauchy problem (\ref{FCP}) is
well-posed if and only if $A$ generates an $\alpha$-ROF.

\begin{definition}
    An $\alpha$-ROF $R_\alpha$ is said exponentially bounded if there exists constants $M\geq 1$ and $w\geq0$ such that
    \begin{equation}\label{eb}
    \|R_\alpha(t)\|\leq M \text{e}^{wt} \quad (t\geq 0).
    \end{equation}
\end{definition}

Along this paper we shall consider only exponentially bounded $\alpha$-ROFs.

\begin{lemma}(\cite[Theorem 2.9]{B}) The operator $A$ generates an $\alpha$-ROF $R_\alpha$ satisfying condition (\ref{eb}) if and only if $(w^\alpha,+\infty) \subset \rho(A)$ and
\begin{equation}
\label{rmfrf} \lambda^{\alpha -1} (\lambda^\alpha-A)^{-1}x = \int_0^\infty e^{-\lambda t}R_\alpha(t)x\, dt \quad (\lambda>w,\,x\in X).
\end{equation}
\end{lemma}
In \cite[Lemma 3]{LMPP} it is shown that if $A$ generates an exponentially bounded $\alpha$-ROF, then $A$ is densely defined.

It is well known that analyticity in the case $\alpha=2,$ that is,  cosine operator functions, implies boundedness of the generator of the family. Therefore we shall restrict ourselves to the case $0<\alpha<2.$
We shall use the following characterization of  generators of  bounded analytic $\alpha$-ROFs \cite[Corollary 2.16]{B}.

\begin{lemma}
\label{Generation_ba}Let $\alpha \in (0,2)$ and $\theta _{0}\in (0,\min \{%
\frac{\pi }{2},\frac{\pi }{\alpha }-\frac{\pi }{2}\}]$. Assume that $A$ is a densely defined operator. The following
assertions are equivalent.

\begin{enumerate}[(i)]
\item $A\in \mathcal{A}_{\alpha }(\theta _{0}).$

\item $\Sigma _{\alpha (\frac{\pi }{2}+\theta _{0})}\subset \rho (A)$
and for each $\theta \in (0,\theta _{0})$ there exists a constant $M_{\theta
}$ such that%
\begin{equation*}
\| z(z-A)^{-1}\| \leq M_{\theta }\;\;(z\in \Sigma _{\alpha (%
\frac{\pi }{2}+\theta)}).
\end{equation*}

\item $-A\in \text{Sect}(\pi -(\frac{\pi }{2}+\theta _{0})\alpha )$.
\end{enumerate}
\end{lemma}

As in the case of bounded analytic $C_0-$semigroups, bounded analytic $\alpha$-ROF have a nice behaviour at infinity. Using the argumentation in \cite[Proposition 2.15]{B}, one can prove the following result.

\begin{lemma}\label{pa} Let $0<\alpha<2$ and $\theta _{0}\in (0,\pi /2]$. Assume that $A\in \mathcal{A}_{\alpha }(\theta _{0})$ and denote by $R_{\alpha}$ the $\alpha$-ROF generated by $A.$ Then

\begin{enumerate}[(i)]
  \item  $A\,R_{\alpha}(t)\in B(X)$, and
\begin{equation*}
\| AR_{\alpha}(t)\| \leq C\,t^{-\alpha}\;\;(t>0).
\end{equation*}
In particular,
\begin{equation*}
\| R_{\alpha}(t)x\| \leq C\,t^{-\alpha}\| y\| \;\;(t>0,\,y\in D(A), x=Ay).
\end{equation*}
  \item  $\lim_{t\rightarrow +\infty} R_{\alpha}(t)x=0$ $(x\in \overline{R(A)}).$
  \item $\| \frac{d^{n}}{dt^{n}}R_{\alpha}(t)\| \leq L t^{-n} \;\; (t>0,\, n\in \mathbb{N}).$
\end{enumerate}

\end{lemma}

Following \cite{CL}, we introduce the notion of integrated fractional resolvent operator function. It can be viewed as a generalization of the notions of integrated semigroups or integrated cosine operator functions.

\begin{definition}
Let $\alpha>0,\,\beta\geq 0.$  An operator $A$ is called the generator of a $\beta$-times integrated $\alpha$-resolvent operator family (for short an $(\alpha,\beta)$-ROF) if there exists a constant $\omega\geq 0$ such that $(\omega^\alpha,+\infty) \subset \rho(A)$ and there exists a strongly  continuous mapping $R_{\alpha,\beta}:[0,+\infty) \rightarrow B(X)$  satisfying
$  \|R_{\alpha,\beta}(t)\|\leq M \text{e}^{\omega t},$ $t\geq 0,$  for a suitable constant $M\ge0,$ such that
\begin{equation}\label{ris}
\lambda^{\alpha-\beta -1} (\lambda^\alpha -A)^{-1}x =\int_0^\infty e^{-\lambda t}R_{\alpha,\beta}(t)x\, dt \quad (x\in X),
\end{equation}
for all $\lambda >\omega.$ In this case $R_{\alpha,\beta}$ is the $(\alpha,\beta)$-ROF generated by $A.$
\end{definition}

By the uniqueness theorem of the Laplace transforms $\{R_{\alpha,\beta}(t)\}_{t\geq0}$ is uniquely determined. The case $\beta=0$ is consistent with the definition of the generator of an $\alpha$-ROF, thus we shall use the notation $R_{\alpha}\equiv R_{\alpha,0}$. Note that $(1,\beta)$-ROFs are $\beta-$times integrated semigroups, $(2,\beta)$-ROFs are $\beta$-times integrated cosine operator functions, and in the particular case of $(2,1)$-ROFs we have sine operator functions. Concrete examples of differential operators that are generators of integrated semigroups or integrated cosine operator functions can be found in \cite{ABHN,H,XL}. We refer the reader to \cite{LS} for the connection between  integrated fractional resolvent operator functions and fractional Cauchy problems.

By means of the Laplace transform it is easy to show that if $A$ is the generator of a $(\alpha,\beta)$-ROF $R_{\alpha,\beta}$, then for $\gamma >\beta$ the operator $A$ generates the $(\alpha,\gamma)$-ROF given by
$ R_{\alpha,\gamma}(t)x=(I_{0+}^{\gamma-\beta}R_{\alpha,\beta}(\cdot)x)(t) \; (t \geq 0, x\in X).$ To our knowledge the situation of $A$ generating an $(\alpha,\delta)$-ROF for $\delta<\beta$ has not been clarified yet (see \cite[Lemma 2.4]{ANS} for the $k$-times integrated semigroup case for integer $k$). We include here the following characterization.

\begin{proposition}\label{mreg}
 Let $\alpha>0,$ $0\leq \delta<\beta,$ $\eta=\beta-\delta$ and $m=\lceil\eta\rceil$. Let $A$ be the generator of an $(\alpha,\beta)$-ROF $R_{\alpha,\beta}$. Then the following are equivalent:
 \begin{enumerate}[(i)]
 \item $A$ is the generator of an $(\alpha,\delta)$-ROF.
 \item There exists an exponentially bounded, strongly continuous mapping $R:[0,+\infty) \rightarrow B(X)$ such that  $R_{\alpha,\beta}(s)x=(I_{0+}^{\eta}R(t)x)(s), \, s \geq 0, \, x \in X$.
 \item For all $x\in X,$  $f(t):=R_{\alpha,\beta}(t)x\in C^{m-1} ([0, \infty); X)$, $I_{0+}^{m-\eta}f\in C^{m} ([0, \infty); X)$, $\mathbf{D}_{t}^{\eta}f:[0,+\infty) \rightarrow B(X)$ is exponentially bounded, and $f^{(k)}(0)=0$ for any $k=0,1,\ldots,m-1.$
 \end{enumerate}
 When these assertions hold, the $(\alpha,\delta)$-ROF generated by $A$ is given by
    \begin{equation}\label{delttis}
    R_{\alpha,\delta}(s)x=(\mathbf{D}_{t}^{\eta}R_{\alpha,\beta}(t)x)(s) \quad (s \geq 0, \, x \in X).
    \end{equation}
\end{proposition}

\begin{proof}
(i)$\rightarrow$(ii) Suppose $A$ is the generator of a $(\alpha,\delta)$-ROF $R_{\alpha,\delta}.$ Then $A$ is the generator of the $(\alpha,\beta)$-ROF $ I_{0+}^{\eta}R_{\alpha,\delta}.$ By the uniqueness of the $(\alpha,\beta)$-ROF generated by $A$ we obtain
$$R_{\alpha,\beta}(s)x=(I_{0+}^{\eta}R_{\alpha,\delta}(t)x)(s) \quad ( s \geq 0,\, x \in X).$$

(ii)$\rightarrow$(iii) It is a straightforward consequence of $f(t)$ being a fractional integral of order $\eta$.

(iii)$\rightarrow$(i) From the fact that $f\in C^{m-1} ([0, \infty); X)$ joint the conditions $f^{(k)}(0)=0$ for any $k=0,1,\ldots,m-1,$ it is not hard to show that
$$(I_{0+}^{m-\eta}f)^{(k)}(0)=0 \quad (k=0,1,\ldots,m-1).$$
Thus, in this case, the fractional integration is an operation inverse to the fractional differentiation from the left (\cite[Lemma 2.5 (b)]{KST}) so that
$$ R_{\alpha,\beta}(s)x=(I_{0+}^{\eta}\mathbf{D}_{t}^{\eta}R_{\alpha,\beta}(t)x)(s) \quad (s \geq 0,\, x\in X). $$
Then taking Laplace transforms in both sides of this equality we obtain
    $$\lambda^{\alpha-\beta-1}(\lambda^\alpha -A)^{-1}x = \lambda^{-\eta} \int_0^\infty e^{-\lambda t}\mathbf{D}_{t}^{\eta}R_{\alpha,\beta}(t)x\, dt \quad (x\in X),$$
for $\lambda $ large enough. Hence $A$ is the generator of the $(\alpha,\delta)$-resolvent family given by (\ref{delttis}).
\end{proof}

We shall use the following result which relates the $(\alpha,\beta)$-ROF and its generator.
\begin{lemma} (\cite[Proposition 3.10]{CL}) \label{relation_g_family} Let $A$ be the generator of an $(\alpha,\beta)$-ROF $R_{\alpha,\beta}$.
    \begin{enumerate}[(i)]
      \item For all $x\in D(A) $ and for all $t\ge0$ we have $R_{\alpha,\beta}(t)x\in D(A)$ with $AR_{\alpha,\beta}(t)x=R_{\alpha,\beta}(t)Ax,$ and
      \begin{equation*}
        R_{\alpha,\beta}(t)x =\frac{t^\beta}{\Gamma(\beta+1)}x+(I_{0+}^{\alpha}R_{\alpha,\beta}(.)Ax)(t).
      \end{equation*}
      \item For all $x\in X $ and for all $t\ge0$ we have $(I_{0+}^{\alpha}R_{\alpha,\beta}(.)x)(t)\in D(A)$ with
      \begin{equation*}
        A(I_{0+}^{\alpha}R_{\alpha,\beta}(.)x)(t)=R_{\alpha,\beta}(t)x-\frac{t^\beta}{\Gamma(\beta+1)}x.
      \end{equation*}
    \end{enumerate}
\end{lemma}

The following result is obtained by an standard argument from Lemma \ref{relation_g_family} and so we omit the proof.
\begin{lemma}
\label{Resolvent_Laplace}Assume that $A$ is the generator of an $(\alpha,\beta)$-ROF $R_{\alpha,\beta }$ verifying  (\ref{eb}). Then for all complex number $\lambda$ with $\text{Re}\lambda>w $ we have  $\lambda ^{\alpha }\in \rho (A)$ and (\ref{ris}) holds.
\end{lemma}

Reasoning as in \cite[Theorem 2.6]{B}, taking into account Lemma \ref{Resolvent_Laplace}  and  the value of the Laplace transform of power function, $\int_0^\infty e^{-\lambda t}t^\beta\, dt=\Gamma(\beta+1)\lambda^{-(\beta+1)},$ we obtain the following generalization to integrated fractional resolvent operator functions of the fact that if $A$ generates an exponentially bounded $\alpha$-ROF for $\alpha>2,$ then $A\in B(X).$ 

\begin{proposition}Let $\alpha>2$, $\beta\geq 0.$ Assume that $A$ is densely defined and that it is the generator of an $(\alpha,\beta)$-ROF $R_{\alpha,\beta}$ satisfying
\begin{equation}\label{ed}
   \| R_{\alpha,\beta}(t) \| \le C t^{\beta} \text{e}^{\omega t} \quad  (t \ge 0)
\end{equation}
for some constants $C,\, \omega \geq 0.$  Then $A\in B(X).$
\end{proposition}

In the following we shall use the classical Mittag-Leffler functions (\cite{KST}), $E_{\alpha,\beta}(z),$ defined by
\begin{equation*}
  E_{\alpha,\beta}(z)=\sum_{k=0}^\infty\frac{z^k}{\Gamma(\alpha k+\beta)}\quad (z\in\mathbb{C}; \alpha,\beta>0).
\end{equation*}
Let $0<\alpha \leq2$ and $\beta\geq 0.$  It is well known that $R_\alpha(t):=E_{\alpha,1}(-t^\alpha) \,I$ is the $\alpha$-ROF generated by $A=-I.$ Then
\begin{equation*}
R_{\alpha,\beta}(t):=I_{0+}^\beta R_\alpha(t)=t^\beta E_{\alpha,\beta+1}(-t^\alpha) \,I,
\end{equation*}
is the $(\alpha,\beta)$-ROF generated by $A.$ From the asymptotic behaviour of Mittag-Leffler functions, we have
\[
   \| R_{\alpha,\beta}(t) \| =   \begin{cases}
                                    O( t^{\beta-\alpha}) & \text{if } \beta+1\ne \alpha,\; \alpha< 2\\
                                    O( t^{-\alpha-1} ) &   \text{if }1<\beta+1=\alpha< 2\\
                                    O( 1+t^{\beta-2} ) & \text{if } \alpha= 2,
                                \end{cases}
\]
as $t \rightarrow \infty.$ In the following proposition we show that such decay rate is optimal for $(\alpha, \beta)$-ROF
when  $\alpha < 2$ and $(\alpha, \beta)\not=(1,0)$.  In the case of $(\alpha, \beta)=(1,0)$, it is well-known that
a $C_0$-semigroup can decay exponentially.

\begin{proposition}\label{asymp}
Let $0<\alpha<2.$ Let $\beta\geq 0$ for $\alpha\neq 1,$ and $\beta> 0$ when $\alpha=1.$ Assume that $A$ generates an $(\alpha,\beta)$-ROF $R_{\alpha,\beta}$ such that
 \[
   \| R_{\alpha,\beta}(t) \| =   \begin{cases}
                                    O( t^{\beta-\alpha-\varepsilon}) &  \text{if } \beta+1 \neq \alpha \\
                                    O( t^{-\alpha-1-\varepsilon} ) &  \text{if } \beta+1=\alpha,
                                \end{cases}
\]
for some small enough $\varepsilon>0,$ as $t \rightarrow \infty.$ Then $X=\{0\}.$
\end{proposition}
\begin{proof}
Let $x\in X$. Define
\begin{equation*}
h(\lambda)=\int_0^\infty  e^{-t\lambda } R_{\alpha,\beta}(t)x \, dt=\lambda^{\alpha-\beta-1}(\lambda^\alpha-A)^{-1}x \quad (\lambda>0).
\end{equation*}
Now we distingue three cases.

(i) $\beta>\alpha.$ Suppose $\| R_{\alpha,\beta}(t) \| =O( t^{\beta-\alpha-\varepsilon})$ as $t \rightarrow \infty$ for some $0<\varepsilon<\beta-\alpha.$ Then there exists a constant $M>0$ such that
\begin{equation*}
  \| R_{\alpha,\beta}(t)\| \leq M(1+t^{\beta-\alpha-\varepsilon}) \quad (t\geq0).
\end{equation*}
Hence,
\begin{equation*}
  \|(\lambda^\alpha-A)^{-1}x\|\leq c\lambda^{\beta+1-\alpha}(\lambda^{-1}+\lambda^{\alpha+\varepsilon-\beta-1})=c(\lambda^{\beta-\alpha}+\lambda^{\varepsilon}),
\end{equation*}
for a suitable constant $c>0$, which shows that
\begin{equation*}
      \lim_{\lambda\rightarrow 0}(\lambda^\alpha  -A)^{-1}x=0.
\end{equation*}
Consequently, $D(A)=\{0\}$ for $x$ is arbitrary. As $\rho(A)\neq\emptyset,$ then $X=\{0\}.$

(ii) $\beta\leq \alpha$ with $\beta+1 \neq \alpha.$ Assume that $\| R_{\alpha,\beta}(t) \| =O( t^{\beta-\alpha-\varepsilon})$ as $t \rightarrow \infty$ for some $\varepsilon>0$ such that $\beta-\alpha-\varepsilon>-2.$  As we also have $\| R_{\alpha,\beta}(t) \| =O( t^{(\beta-\alpha-\varepsilon)/2})$ as $t \rightarrow \infty,$ we obtain that $\lambda^{(2+\beta-\alpha-\varepsilon)/2} h(\lambda)$ and $\lambda^{2+\beta-\alpha-\varepsilon} h'(\lambda)$ are bounded in $\lambda>0$. Therefore, from the identity
\begin{equation*}
 h'(\lambda)=(\alpha-\beta-1) \lambda^{\alpha-\beta-2}(\lambda^\alpha-A)^{-1}x-\alpha \lambda^{2(\alpha-1)-\beta}(\lambda^\alpha-A)^{-2}x,
\end{equation*}
we have
\begin{eqnarray*}
         \|(\alpha-1-\beta)(\lambda^\alpha-A)^{-1}x\|&=& \lambda^{2+\beta-\alpha}\|h'(\lambda)+\alpha \lambda^{2(\alpha-1)-\beta}(\lambda^\alpha-A)^{-2}x\| \\
         &\leq& c \lambda^{\varepsilon}(1+\lambda^{\beta}),
\end{eqnarray*}
for a suitable constant $c>0$. This implies as in case (i) that $X=\{0\}.$

(iii) $\beta+1 =\alpha.$ Suppose $\| R_{\alpha,\beta}(t) \| =O( t^{-\alpha-1-\varepsilon})$ as $t \rightarrow \infty,$  for some $\varepsilon>0$ verifying that $\alpha+\varepsilon<2.$ Then $\lambda^{(2-\alpha-\varepsilon)/3} h(\lambda)$ and $\lambda^{2-\alpha-\varepsilon} h''(\lambda)$ are bounded in $\lambda>0$. Since now
\begin{equation*}
      h''(\lambda)=-\alpha (\alpha-1) \lambda^{\alpha-2}(\lambda^\alpha-A)^{-2}x+2 \alpha^2 \lambda^{2(\alpha-1)}(\lambda^\alpha-A)^{-3}x,
\end{equation*}
we get
\begin{eqnarray*}
         \|\alpha (\alpha-1)(\lambda^\alpha-A)^{-2}x\|&=& \lambda^{2-\alpha}\|2 \alpha^2 \lambda^{2(\alpha-1)}(\lambda^\alpha-A)^{-3}x-h''(\lambda)\| \\
         &\leq& c \lambda^{\varepsilon}(\lambda^{2(\alpha-1)}+1),
      \end{eqnarray*}
for a suitable constant $c>0$. As $\alpha>1,$ we obtain that $D(A^2)=\{0\},$ which yields $X=\{0\}.$
\end{proof}

It is a well-known fact that a cosine operator function necessarily has non-negative exponential type (\cite[Proposition 3.14.6]{ABHN}). We have the following result for integrated cosine operator functions.

\begin{proposition}\label{asympc}
Let $\beta>0.$ Assume that $A$ generates a $(2,\beta)$-ROF $R_{2,\beta}$ such that
 \[
   \| R_{2,\beta}(t) \| =   \begin{cases}
                                    O( t^{\beta-2-\varepsilon}) &  \text{if } \beta\neq 1\\
                                    O( t^{-\varepsilon} ) &  \text{if } \beta=1,
                             \end{cases}
\]
for some small enough $\varepsilon>0,$ as $t \rightarrow \infty.$ Then $X=\{0\}.$
\end{proposition}
\begin{proof} The case $\beta>2$ runs as in (i) of Proposition \ref{asymp}. When $\beta\leq 2$ and $\beta\neq1,$ one can reason as in (ii) of Proposition \ref{asymp}.

Suppose that $\beta=1.$ It follows from \cite{LSh} that
 \begin{equation}\label{fe_sine}
  2R_{2,1}(t)R_{2,1}(s)x=\int_{t-s}^{t+s}R_{2,1}(\tau)x\, d\tau \quad (t\ge s \ge 0, x\in X).
\end{equation}
Assume that $\| R_{2,1}(t) \| =  O( t^{-\varepsilon})$ as $t\rightarrow\infty$, for some $0<\varepsilon<1.$ By (\ref{fe_sine}) we have
\begin{equation}
\|\int_{0}^{t}R_{2,1}(s)x\, ds\| =2\|R_{2,1}(\frac{t}{2})^2x\|\leq c \|x\|t^{-2\varepsilon} \quad (t>0, x\in X),
\end{equation}
for a suitable constant $c>0.$ As $\int_{0}^{t}R_{2,1}(s)\, ds$ is a $(2,2)$-ROF, this yields that $X=\{0\}.$
\end{proof}

\begin{remark}\label{rasymp}
  From Proposition \ref{asymp} and Proposition \ref{asympc} it follows that if $R_{\alpha,\beta}$ is an $(\alpha,\beta)$-ROF, $0<\alpha \leq2$ and $\beta>\alpha,$ then $R_{\alpha,\beta}(t)$ is not uniformly bounded, that is, there exists a sequence  $\{t_n\}_{n=1}^\infty$ such that  $\lim_{n\rightarrow\infty}t_n=\infty$ and $\lim_{n\rightarrow\infty}\|R_{\alpha,\beta}(t_n)\|=\infty.$
\end{remark}

In \cite[Proposition 2.7]{LZ} it is shown that and $\alpha$-ROF is never exponentially stable when $0<\alpha<1$. As an straightforward consequence of our analysis we extend these result to $(\alpha,\beta)$-ROFs, by showing that only $C_0$-semigroups can decrease exponentially.

\begin{corollary}
Let $0< \alpha \le 2$, $\beta\geq 0.$ Suppose that $A$ is the generator of an $(\alpha,\beta)$-ROF $R_{\alpha,\beta}.$ Then $R_{\alpha,\beta}$  is never exponentially stable but for $(\alpha, \beta)=(1,0).$
\end{corollary}

As the function $f(z)=z^{-1}$ preserves sectors, that is, $f(\Sigma _{\theta })=\Sigma _{\theta },$ it is natural to consider the inverse problem for generators of bounded fractional resolvent operator functions. In the more general setting of integrated fractional resolvent operator functions, the natural condition is generating the so called tempered operator functions (see Proposition \ref{mreg}).

\begin{definition}
Let $0<\alpha\leq 2,$ and $\beta\geq 0.$  Then an operator $A$ is called the generator of a tempered $(\alpha,\beta)$-ROF if $(0,+\infty) \subset \rho(A)$ and there exists a strongly  continuous mapping $R_{\alpha,\beta}:[0,+\infty) \rightarrow B(X)$  satisfying
\begin{equation}\label{tempered}
 \|R_{\alpha,\beta}(t)\|\leq M t^{\beta} \quad (t\geq 0),
\end{equation}
for a suitable constant $M\ge0,$ such that (\ref{ris}) is satisfied for all $\lambda >0.$
\end{definition}

In \cite[Example 6.10]{A} an example of once integrated semigroup whose generator is densely defined and which is not of $O(t)$ as $t\rightarrow 0$ is provided. In \cite[Theorem 4.4]{AK} it is shown that generating smooth distribution semigroups is equivalent to generating tempered integrated semigroups. A characterization of operators $iA$ generating tempered integrated groups, where $A$ generates a holomorphic semigroup of angle $\pi/2,$ can be found in \cite{M}. It is shown in \cite{H} that this optimal convergence rate is achieved for a large class of differential operators. We consider here the following situation.

 Let $A=\sum_{|\alpha|\leq m}c_\alpha D^\alpha$ be a linear differential operator of order $m>1$ with constant coefficients on one of the spaces $L^p(\mathbb{R}^n),$ $1\leq p<\infty$, or $C_0(\mathbb{R}^n),$ with maximal domain in the sense of the Fourier transform. Assume that the symbol of $A$, defined as $P(D)=\sum_{|\alpha|\leq m}c_\alpha (i\xi)^\alpha$, is of the form $ia(\xi)$ where $a:\mathbb{R}^n \mapsto \mathbb{R}$ is an homogeneous polynomial which is elliptic ($|a(\xi)|\geq C |\xi|^m $ for all $\xi \in \mathbb{R}^n$ with $|\xi|>L$). Note that $A$ is an injective densely defined operator. Moreover $A$ has dense range but for $p=1$. It is interesting to observe that in the particular case of $a(\xi)=-|\xi|^2$ we obtain the Schr\"{o}dinger operator $i\Delta.$ The following result is taken from \cite[Theorem 4.2]{H}.

\begin{lemma} \label{do} The operator $A$ on  $L^p(\mathbb{R}^n),$ $1< p<\infty,$ $p\neq2,$ (respectively, $L^1(\mathbb{R}^n)$  or $C_0(\mathbb{R}^n)$) is the generator of a tempered $\beta$-times integrated semigroup for all $\beta \geq n|\frac{1}{2}-\frac{1}{p}|$ (respectively, $\beta > \frac{n}{2}$).
\end{lemma}

In \cite{H3} we can also find examples of differential operators on $L^p(\mathbb{R}^n)^N,$ $1< p<\infty,$ generating tempered integrated semigroups.

Examples of differential operators generating tempered integrated cosine operator functions can be found in \cite{AK,H,MK}. We recall that the Laplacian operator $\Delta$ on $L^p(\mathbb{R}^n),$ $n\geq 2,$ $1\leq p<\infty,$ $p\neq2$, with maximal distributional domain, is the generator of a $\beta$-times integrated cosine operator function for any $\beta > (n-1)|\frac{1}{2}-\frac{1}{p}|$ (\cite[Proposition 3.2]{MK}). Notice that the  Laplacian operator is injective densely defined. Moreover it has dense range but for $p=1.$


Now we study the behaviour of tempered $(\alpha,\beta)$-ROFs as $t\rightarrow0$ and as $t\rightarrow \infty.$

\begin{proposition} \label{batoatinf} Let $0<\alpha \leq 2$ and $\beta\geq 0.$ Assume that $A$ is the generator of a tempered $(\alpha,\beta)$-ROF $R_{\alpha,\beta}.$ The following assertions hold.

\begin{enumerate}[(i)]
\item $x\in \overline{D(A)}$ if and only if
\begin{equation} \label{convat0}
  \lim_{t\rightarrow 0^+} \frac{R_{\alpha,\beta}(t)x}{t^\beta}=\frac{x}{\Gamma(\beta +1)}.
\end{equation}

 \item $x\in \overline{R(A)}$ if and only if
 \begin{equation} \label{convatinf}
 \lim_{t\rightarrow \infty} \frac{(I_{0+}^{\alpha}R_{\alpha,\beta}(.)x)(t)}{t^{\alpha+\beta}}=0.
 \end{equation}
 \end{enumerate}
\end{proposition}

\begin{proof}
(i) Let $x\in X.$ By (ii) of Lemma \ref{relation_g_family} we have $(I_{0+}^{\lceil \alpha \rceil}R_{\alpha,\beta}(.)x)(t)\in D(A)$ and hence taking derivatives of integer order it follows that $R_{\alpha,\beta}(t)x\in \overline{D(A)}.$ Now it is clear that (\ref{convat0}) implies that $x\in \overline{D(A)}.$

It is not hard to show that there exists a constant $M>0$ such that
\begin{equation}\label{ubound}
\|(I_{0+}^{\alpha}R_{\alpha,\beta})(t)\|\leq M t^{\alpha+\beta}\quad (t>0).
\end{equation}
From (i) of Lemma \ref{relation_g_family} we obtain that (\ref{convat0}) holds on $D(A),$ and hence it also holds for   $x\in \overline{D(A)}.$

(ii) By (i) of Lemma \ref{relation_g_family} it follows that (\ref{convatinf}) holds on $R(A),$ and thanks to (\ref{ubound}) it also holds for   $x\in \overline{R(A)}.$ Conversely, suppose (\ref{convatinf}) holds.  From the fact that $I_{0+}^{\alpha}R_{\alpha,\beta}$ is the $(\alpha,\alpha+\beta)$-ROF generated by $A$ and (ii) of Lemma \ref{relation_g_family},  we have
\begin{equation*}
A(I_{0+}^{2\alpha}R_{\alpha,\beta}(.)x)(t)=(I_{0+}^{\alpha}R_{\alpha,\beta}(.)x)(t)-\frac{t^{\alpha+\beta}}{\Gamma(\alpha+\beta+1)}x,
\end{equation*}
and accordingly $x\in \overline{R(A)}.$
\end{proof}

From Lemma \ref{Resolvent_Laplace} it is easy to obtain the following generalization of \cite[Lemma 2.6]{LCL}  to the case of tempered integrated resolvent operator functions.

\begin{proposition} \label{Generation_bounded}
Let $0<\alpha \leq 2$ and $\beta\geq 0.$ The operator $A$ is the generator of a tempered $(\alpha,\beta)$-ROF if and only if $\Sigma_{\alpha \pi/2}\subset \rho(A)$ and there exists a strongly continuous function  $R_{\alpha,\beta}:[0,+\infty)\rightarrow B(X)$ satisfying  (\ref{tempered}) such that (\ref{ris}) holds for all $\text{Re}\lambda >0.$ In this case, $R_{\alpha,\beta }$ is the $(\alpha,\beta)$-ROF generated by $A$.
\end{proposition}

\begin{theorem} \label{subordination_bounded}
Let $0<\alpha \leq 2$ and $\beta\geq 0.$  Assume that $A$ is the generator of a tempered $(\alpha,\beta)$-ROF. The following assertions hold.
\begin{enumerate}[(i)]
  \item $-A\in \text{Sect}(\pi-\alpha \pi/2).$
  \item Suppose $A$ is densely defined. Then $A\in \mathcal{A}_{\gamma }(\theta _{\gamma})$ for all $0<\gamma<\alpha,$ where $\theta _{\gamma}=\min \{1,\frac{\alpha-\gamma}{\gamma }\}\frac{\pi }{2}$.
\end{enumerate}

\end{theorem}

\begin{proof}(i) Thanks to Proposition \ref{Generation_bounded}, for $r>0$ and $0\leq \varphi <\alpha \pi/2$ we obtain that
$$\|(r\text{e}^{i\varphi}-A)^{-1} \| \leq \Gamma(\beta+1)M \cos(\varphi/\alpha)^{-\beta-1} r^{-1},$$
and from it the result follows easily.

(ii) It is a straightforward consequence of (i) and Lemma \ref{Generation_ba}.
\end{proof}

\begin{remark}
Recently a subordination principle for regularized resolvent families has been obtained in \cite[Theorem 4.5]{AM} (see also \cite[Theorem 4.13]{KLW}). However, the subordination principle (ii) of Theorem \ref{subordination_bounded} does not follow from it, although in the case $\beta=0$ it can be obtained from Bajlekova's subordination principle (\cite[Theorems 3.1 and 3.3]{B}).
\end{remark}

\setcounter{equation}{0}

\section{The inverse generator problem}\label{MS}
We start this section by analysing the inverse generator problem for generators of analytic resolvent operator functions. From \cite[Theorem 4.9]{CL} we know that an operator is the generator of a tempered analytic $(\alpha,\beta)$-ROF if and only if it is the generator of a bounded analytic $\alpha$-ROF. Therefore we shall only consider the later in our study.
\begin{theorem}\label{inv_generator}
Assume that $A$ is a linear operator with dense range. The following assertions hold.
\begin{enumerate}[(i)]
\item  Let $\alpha \in (0,2)$ and $\theta _{0}\in (0,\min \{\frac{\pi }{2%
},\frac{\pi }{\alpha }-\frac{\pi }{2}\}]$. If $A\in \mathcal{A}_{\alpha
}(\theta _{0})$, then $A^{-1}\in \mathcal{A}_{\alpha }(\theta _{0}).$
\item  Let $0<\gamma <\alpha \leq 2$ and $\beta\geq 0$. If $A$ is the generator of a tempered $(\alpha,\beta)$-ROF $R_{\alpha,\beta;A}$ and it is densely defined, then
$A^{-1}\in \mathcal{A}_{\gamma }(\phi_\gamma )$, where $\theta _{\gamma}=\min \{1,\frac{\alpha-\gamma}{\gamma }\}\frac{\pi }{2}$. Moreover, the $\gamma$-ROF generated by $A^{-1}$ is given by
\begin{equation}\label{sub_irof}
  \begin{aligned}
    &R_{\gamma;A^{-1}}(t)x \\
    &= x-t^{\gamma(1+\beta)/\alpha} \int_0^\infty   \phi(\gamma/\alpha,1+\gamma(1+\beta)/\alpha;-st^{\gamma/\alpha})  R_{\alpha,\beta;A}(s)x \,ds
  \end{aligned}
\end{equation}
for all $x\in X,\,t>0.$
\end{enumerate}
\end{theorem}

\begin{proof} (i) It is a consequence of the nice fact that the inverse of a sectorial operator is sectorial of the same angle and Lemma \ref{Generation_ba}.

(ii) It follows from Theorem \ref{subordination_bounded} and (i). Let us show (\ref{sub_irof}). From (\ref{b_W}) it is easy to show that $R_{\gamma;A^{-1}}(t)x$  is continuous for $t>0,$ and it is bounded since
\begin{eqnarray*}
 \| R_{\gamma;A^{-1}}(t)x -x \|& \leq & \| x \|L \,  t^{\gamma(1+\beta)/\alpha} \int_0^\infty   \exp(-as^qt^{q\gamma/\alpha})\,  s^\beta ds \\
   & \leq & \| x \|\, L \,q^{-1}\, t^{\gamma(1+\beta)/\alpha} \int_0^\infty   \exp(-art^{q\gamma/\alpha})\,  r^{(\beta+1)/q-1} dr \\
   & = & \| x \| \,L\, q^{-1}\, \Gamma((\beta+1)/q) \, a^{-(\beta+1)/q}.
\end{eqnarray*}
Hence, $R_{\gamma;A^{-1}}(t)x$ is Laplace transformable for $\text{Re}\lambda>0.$ By applying Fubini's theorem and by (\ref{LTW}), we get that
 \begin{equation*}
  \begin{aligned}
    &\int_0^\infty  e^{-t\lambda } R_{\gamma;A^{-1}}(t)x \, dt=\lambda^{-1}x\\
 & \hspace{1mm} - \int_0^\infty R_{\alpha, \beta;A}(s)x  \int_0^\infty e^{-t \lambda } t^{\gamma(1+\beta)/\alpha}  \phi(\gamma/\alpha,1+\gamma(1+\beta)/\alpha,-st^{\gamma/\alpha})\, dt \, ds \\
   &=\lambda^{-1}x- \lambda^{-1-\gamma(1+\beta)/\alpha}\int_0^\infty \exp(-s\lambda^{-\alpha/\gamma}) R_{\alpha,\beta;A}(s)x  \, ds \\
   &= \lambda^{-1}x-\lambda^{-1-\gamma}(\lambda ^{-\gamma }-A)^{-1}x\\
   &= \lambda^{\gamma-1}(\lambda ^{\gamma }-A^{-1})^{-1}x.
   \end{aligned}
\end{equation*}
which shows our claim thanks to the uniqueness of the Laplace transform.
\end{proof}

\begin{remark}
 Notice that by (ii) of Lemma \ref{pa}  all the families generated by the inverse operator in Theorem \ref{inv_generator} are stable.
\end{remark}

The next theorem is a generalization of \cite[Theorem 3.3]{PZ},  which shows that the inverse of the generator of a bounded $C_0-$semigroups always generates a once integrated semigroup, to the more general setting of inverses of generators of tempered integrated fractional resolvent families.

\begin{theorem}\label{MT}
  Let $0<\alpha\leq 2,$  $\beta\geq 0$ and $\gamma>\beta+1/2.$ Let $A$ be the generator of a tempered $(\alpha,\beta)$-ROF $R_{\alpha,\beta;A}$. Assume that $A$ is injective. Then the following assertions hold:
   \begin{enumerate}[(i)]
     \item The inverse operator $A^{-1}$ is the generator of a tempered $(\alpha,\gamma)$-ROF $R_{\alpha,\gamma;A^{-1}}$ given by
     \begin{equation}\label{IRF}
     \begin{aligned}
     R&_{\alpha,\gamma;A^{-1}}(t)x \\
     &= \frac{t^\gamma}{\Gamma(\gamma+1)} x-t^{\frac{1+\beta+\gamma}{2}} \int_0^\infty  J_{1+\beta+\gamma}(2\sqrt{st}) s^{-\frac{1+\beta+\gamma}{2}}  R_{\alpha,\beta;A}(s)x \,ds
     \end{aligned}
     \end{equation}
     for all $x\in X,\,t>0.$ 
     \item The operator $A^{-1}$ is the generator of a tempered $(\alpha,\beta)$-ROF if and only if  for all $x\in X,$  the mapping $R_{\alpha,1+\beta;A^{-1}}(\cdot)x:[0, \infty)\rightarrow X$ is continuously differentiable and its derivative  $\frac{d}{dt}R_{\alpha,1+\beta;A^{-1}}:[0,\infty) \rightarrow B(X)$ satisfies condition (\ref{tempered}). In this case, the tempered $(\alpha,\beta)$-ROF generated by $A^{-1}$ is given by
         \begin{equation*}
         R_{\alpha,\beta;A^{-1}}(s)x=\frac{d}{dt}(R_{\alpha,1+\beta;A^{-1}}(t)x)(s) \quad (s \geq 0, \, x \in X).
         \end{equation*}
   \end{enumerate}
\end{theorem}
   \proof (i) For all $x\in X$ and $t>0,$ as
   $$R_{\alpha,\gamma;A^{-1}}(t)x= \frac{t^\gamma}{\Gamma(\gamma+1)} x-t^{\beta+\gamma} \int_0^\infty  J_{1+\beta+\gamma}(2\sqrt{s}) s^{-\frac{1+\beta+\gamma}{2}}  R_{\alpha,\beta;A}(s/t)x \,ds,$$
   the operator function $R_{\alpha,\beta;A}$ is tempered, and the mapping
   $$r>0\mapsto J_{1+\beta+\gamma}(2\sqrt{r}) r^{-(1+\gamma-\beta)/2}$$
   is integrable in $(0,\infty),$ it follows that
   $$\|R_{\alpha,\gamma;A^{-1}}(t)x\|\leq M t^\gamma \|x\| .$$
   Hence  $R_{\alpha,\gamma;A^{-1}}(t)x$ is a continuous mapping of $t\geq 0$ with limit $0$ at $t=0$. Thus we can take the Laplace transform of the family. Thanks to Proposition \ref{Generation_bounded}, (\ref{LT}) and by applying Fubini's theorem, we obtain for $\text{Re}\lambda>0$ that
    \begin{equation*}
  \begin{aligned}
    &\int_0^\infty  e^{-t\lambda } R_{\alpha,\gamma;A^{-1}}(t)x \, dt\\
   &= \lambda^{-(1+\gamma)}x- \int_0^\infty R_{\alpha,\beta;A}(s)x \,s^{-\frac{1+\beta+\gamma}{2}} \int_0^\infty e^{-t \lambda } t^{\frac{1+\beta+\gamma}{2}} \,J_{1+\beta+\gamma}(2\sqrt{st})\, dt \, ds \\
   &=\lambda^{-(1+\gamma)}x- \lambda^{-(2+\beta+\gamma)} \int_0^\infty \text{e}^{-s/\lambda} R_{\alpha,\beta;A}(s)x  \, ds \\ \nonumber
   &= \lambda^{-(1+\gamma)}x-\lambda^{-(1+\gamma+\alpha)}(\lambda ^{-\alpha }-A)^{-1}x\\ \nonumber
   &= \lambda^{\alpha-\gamma-1}(\lambda ^{\alpha }-A^{-1})^{-1}x.
   \end{aligned}
\end{equation*}
   which shows our assertion. 

   (ii) It is a straightforward consequence of Proposition \ref{mreg}.

%
\endproof
   \begin{remark}
     For simplicity, consider Theorem \ref{MT} for $\beta=0.$ In order to improve the obtained result to $\gamma=1/2$ we should dealt with the function
     \begin{equation*}
       J_\frac{3}{2}(2\sqrt{s})s^{-3/4} = \frac{1}{2\sqrt{\pi}}\left( \sin(\sqrt{2 s})s^{-3/2}-2\cos(2\sqrt{s})s^{-1}\right ),
     \end{equation*}
     which does not belong to $L^1(1,\infty).$

     From the proof of (i) of Theorem \ref{MT} it is clear that the key is the behaviour of the function $R_{\alpha,\beta;A}$ at infinity. Assume that  $A$ is injective and it is the generator of an $(\alpha,\beta)$-ROF $R_{\alpha,\beta;A}$ such that
     \begin{equation}\label{tinf}
     \|R_{\alpha,\beta}(t)\|\leq M (1+t^{\delta}) \quad (t\geq 0),
     \end{equation}
     for some $\delta\geq0$, which is equivalent to say that $\|R_{\alpha,\beta}(t)\|=O(t^{\delta})$ as $t\rightarrow\infty.$ Then for all $\gamma\geq0,$ $\gamma>2\delta+1/2-\beta$, the inverse operator $A^{-1}$ is the generator of an $(\alpha,\gamma)$-ROF $R_{\alpha,\gamma;A^{-1}}$ given by (\ref{IRF}). In particular, if $2\delta+1/2-\beta<0,$ then $A^{-1}$ generates an $\alpha$-ROF. This result is a generalization of \cite[Theorem 4.13]{dL3}.

     As an interesting example of $(\alpha,\beta)$-ROF satisfying (\ref{tinf}), we recall that the Korteweg-De Vries operator $A=\frac{d^3}{dt^3}+\frac{d}{dt}$ on the spaces $X=L^p(\mathbb{R}),$ $1\leq p<\infty$, is injective and generates an $(1,\beta)$-ROF verifying (\ref{tinf}) for $\delta=\beta,$ for all $\beta>|\frac{1}{2}-\frac{1}{p}| $ (\cite{H}).

   \end{remark}

  From Theorem \ref{MT} we obtain easily a representation of the resolvent operator function generated by $A^{-1}$ in the case (i) of Theorem \ref{inv_generator}.

   \begin{corollary} Assume that $A$ generates a bounded analytic $\alpha$-ROF $R_{\alpha;A}$ and that $\overline{R(A)}=X$. Then $A^{-1}$ generates the bounded analytic $\alpha$-ROF $R_{\alpha;A^{-1}}$ given by
    \begin{equation}\label{rana}
     R_{\alpha;A^{-1}}(t)x = x+t^{-1}\int_0^\infty  J_2(2\sqrt{r}) \left( R_{\alpha;A}'(r/t)x -t r^{-1} R_{\alpha;A}(r/t)x\right ) \,dr \\
     \end{equation}
for $t>0,\, x\in X.$
\end{corollary}
  \begin{proof}
  By Theorem \ref{MT} we know that $A^{-1}$ generates an $(\alpha,1)$-ROF $R_{\alpha,1;A^{-1}}$ given by (\ref{IRF}) for $\beta=0$ and $\gamma=1$. From (iii) of Lemma \ref{pa} it is easy to show that the derivative of $R_{\alpha,1;A^{-1}}(t)x$ is given by (\ref{rana}).
  \end{proof}

  To finish the paper we apply the results of this section to some fractional Cauchy problems unsolved for the fractional derivative. Our motivation comes from the fractional version of the Barenblatt-Zheltov-Kochina equation
   \begin{equation*}
    \frac{\partial}{\partial t}(\lambda-\Delta)u(t)= \mu \Delta u(t),
  \end{equation*}
  which models the filtration of the fluid in fissured-porous medium (\cite{BZK,FNG}), and the Boussinesq  equation
  \begin{equation*}
    \frac{\partial^2}{\partial t^2}(\lambda-\Delta)u(t)= \mu \Delta u(t),
  \end{equation*}
  which models longitudinal waves in a thin elastic bar, where $\lambda, \mu>0$ and $\Delta$ is the Laplacian operator (\cite{Wh}). In both cases the equation can be written under the pattern of equation (\ref{e1}).

\begin{corollary}\label{Cp1}
Let $0<\gamma<\alpha\leq 2,$ $m=\lceil\gamma\rceil$ and $\beta>0.$ Let $B\in B(X)$ and $a>0.$ Suppose $A$ is the generator of a tempered $(\alpha,\beta)$-ROF and it has dense domain and range. The fractional Cauchy problem
\begin{eqnarray}\label{e1}
&\mathbf{D}_{t}^\gamma Au(t)= BAu(t)+a u(t), \;\; t>0,  \\
& (Au(t))^{(k)}(0)=u_k, 0\leq k \leq m-1, \label{e2}
\end{eqnarray}
has a unique strong solution, for all initial data $u_k \in R(A);$ that is, a function $u\in C([0,\infty);D(A))$ such that $w:=Au\in C^{m-1}([0,\infty);X))$, $I^{m-\gamma}_{0+}(w-\sum_{k=0}^{m-1}w^{(k)}(0)\,g_{k+1}) \in C^{m}([0,\infty);X)$ and (\ref{e1})-(\ref{e2}) holds.
\end{corollary}
\begin{proof}
  Thanks to (ii) of Theorem \ref{inv_generator} and the additive perturbation by a bounded operator theorem (\cite[Theorem 2.25]{B}), we can assure that $B+a A^{-1}$ is the generator of an analytic $\gamma$-ROF for all $\gamma<\alpha$. Therefore, the fractional Cauchy problem
\begin{equation}\label{te}
\mathbf{D}_{t}^\gamma  v(t)= Bv(t)+a A^{-1}v(t), \; t> 0;   v^{(k)}(0)=v_k, 0\leq k \leq m-1,
\end{equation}
has a unique strong solution for all initial data $v_k \in R(A)$. Let $v(t)$ be the strong solution of (\ref{te}) for  $v_k=u_k, 0\leq k \leq m-1$. This means that $v\in C([0,\infty);R(A))$, $(B+a A^{-1})v\in C([0,\infty);X),$ $v\in C^{m-1}([0,\infty);X)),$ $I^{m-\gamma}_{0+}(v-\sum_{k=0}^{m-1}v^{(k)}(0)\,g_{k+1})\in C^{m}([0,\infty);X)$ and (\ref{te}) is satisfied. Then it is not hard to show that the continuous function $u(t)=A^{-1}v(t),$ $t\geq0,$ satisfies the required properties. Uniqueness of the strong solution of (\ref{e1})-(\ref{e2}) follows easily from uniqueness of the strong solution of (\ref{te}).
\end{proof}

\begin{corollary}\label{Cp2}
Let $\beta\geq 0,$ $\gamma>\beta+1/2,$ $m=\lceil\gamma\rceil$ and $a>0.$ Suppose $A$ is the generator of a tempered $(1,\beta)$-ROF and it is injective. Let $B\in B(X)$ such that $B(1-A)^{-1}=(1-A)^{-1}B.$ The Cauchy problem
\begin{eqnarray}\label{e3}
\frac{d}{dt}Au(t)= BAu(t)+a u(t), \;\; t> 0, u(0)=u_0,
\end{eqnarray}
has a unique strong solution, for all initial data $u_0 \in D(A)\cap R(A^{m});$ that is, a function $u\in C([0,\infty);D(A))$ such that $Au\in C^{1}([0,\infty);X))$  and (\ref{e3}) holds.
\end{corollary}
\begin{proof} Since the class of integrated semigroups is closed under additive bounded perturbations which commutes with the generator (\cite[Theorem 1.3.5]{XL}), then by (i) of Theorem \ref{MT}
  we obtain that $B+a A^{-1}$ is the generator of a $(1,\gamma)$-ROF for all $\gamma>\beta+1/2$. Therefore, the Cauchy problem
\begin{equation}\label{te2}
\frac{d}{dt} v(t)= Bv(t)+a A^{-1}v(t), \; t> 0;   v(0)=v_0,
\end{equation}
has a unique strong solution for all initial data $v_0 \in R(A^{m+1})$ (\cite[Theorem 3.2.13]{ABHN}. Let $v(t)$ be the strong solution of (\ref{te2}) for  $v_0=Au_0;$ i.e., $v\in C([0,\infty);R(A))\cap C^{1}([0,\infty);X)$ and (\ref{te2}) is satisfied. Thus it is easy to show that the continuous function $u(t)=A^{-1}v(t),$ $t\geq0,$ satisfies the required properties. Uniqueness of the strong solution of (\ref{e3}) follows easily from uniqueness of the strong solution of (\ref{te2}).
\end{proof}

   Recall that, in general, additive perturbation by a bounded operator of the generator of an integrated semigroup can fail to generate an integrated semigroup, if the bounded perturbation does not commute with the generator (\cite[Example 3.7]{KH}). Highlight that the set of suitable initial data in Corollary \ref{Cp2} can be widen by using fractional powers of operators (\cite{NS}).

   Notice that Corollary \ref{Cp1} and Corollary \ref{Cp2} can be applied, for instance, to the differential operators introduced in Section \ref{ifrof}, including the Sch\"{o}dinger and the Laplacian operators. \\

\bigskip
\footnotesize
\noindent\textit{Acknowledgments.} This work was done while the second author was visiting the Department of Mathematics of Sichuan University, supported by Universitat de Val\`{e}ncia grant UV-INV\_EPDI17-545952. He would like to thank the functional analysis group, and very specially Professor Miao Li, for their warm hospitality and support. The first author was supported by the NSFC of China (No. 11371263) and NSFC-RFBR Programme of China (No. 11611530677). The third author was supported by the Russian Foundation for Basic Research, projects 15-01-00026\_a,16-01-00039\_a  and   17-51-53008.\\

\end{document}